\documentclass[12pt]{amsart}
\usepackage[T1]{fontenc}
\usepackage{lmodern}
\usepackage{ifthen}
\usepackage{diagrams}
\usepackage{amscd}
\usepackage{amsxtra}
\usepackage{amssymb}
\usepackage{amsthm}
\usepackage{array}
\usepackage[margin=1in]{geometry}
\usepackage{xcolor}
\definecolor{cite}{rgb}{0.50,0.00,1.00}
\definecolor{url}{rgb}{0.00,0.50,0.75}
\definecolor{link}{rgb}{0.00,0.00,0.50}
\usepackage[colorlinks,linkcolor=link,urlcolor=url,citecolor=cite,pagebackref,breaklinks]{hyperref}
\usepackage{mathtools}
\usepackage{mathrsfs}
\usepackage[all]{xy}
\usepackage[lite,abbrev,msc-links]{amsrefs}
\begin{document}
\title[$E_1$-degeneration]{On $E_1$-degeneration for the special fiber of a semistable family}
\author[Mao Sheng]{Mao Sheng}
\email{msheng@ustc.edu.cn}
\author[Junchao Shentu]{Junchao Shentu}
\email{stjc@amss.ac.cn}
\address{School of Mathematical Sciences,
	University of Science and Technology of China, Hefei, 230026, China}


\theoremstyle{definition}
\newtheorem{Unity}{Unity}[section]
\newtheorem*{defn*}{Definition}
\newtheorem{defn}[Unity]{Definition}
\newtheorem{exmp}[Unity]{Example}
\newtheorem{Notation}[Unity]{Notation}
\newtheorem{Claim}[Unity]{Claim}
\newtheorem*{Setting}{Setting}
\newtheorem{rmk}[Unity]{Remark}
\theoremstyle{plain}
\newtheorem*{thm*}{Theorem}
\newtheorem{thm}[Unity]{Theorem}
\newtheorem{prop}[Unity]{Proposition}
\newtheorem*{prop*}{Proposition}
\newtheorem*{cor*}{Corollary}
\newtheorem{cor}[Unity]{Corollary}
\newtheorem{lem}[Unity]{Lemma}
\newtheorem{conj}[Unity]{Conjecture}
\newtheorem{prob}[Unity]{Problem}
\newtheorem{question}[Unity]{Question}
\theoremstyle{remark}

\numberwithin{Unity}{section}

\newcommand{\sA}{\mathscr{A}}
\newcommand{\sB}{\mathscr{B}}
\newcommand{\sC}{\mathscr{C}}
\newcommand{\sD}{\mathscr{D}}
\newcommand{\sE}{\mathscr{E}}
\newcommand{\sF}{\mathscr{F}}
\newcommand{\sG}{\mathscr{G}}
\newcommand{\sH}{\mathscr{H}}
\newcommand{\sI}{\mathscr{I}}
\newcommand{\sJ}{\mathscr{J}}
\newcommand{\sK}{\mathscr{K}}
\newcommand{\sL}{\mathscr{L}}
\newcommand{\sM}{\mathscr{M}}
\newcommand{\sN}{\mathscr{N}}
\newcommand{\sO}{\mathscr{O}}
\newcommand{\sP}{\mathscr{P}}
\newcommand{\sQ}{\mathscr{Q}}
\newcommand{\sR}{\mathscr{R}}
\newcommand{\sS}{\mathscr{S}}
\newcommand{\sT}{\mathscr{T}}
\newcommand{\sU}{\mathscr{U}}
\newcommand{\sV}{\mathscr{V}}
\newcommand{\sW}{\mathscr{W}}
\newcommand{\sX}{\mathscr{X}}
\newcommand{\sY}{\mathscr{Y}}
\newcommand{\sZ}{\mathscr{Z}}
\newcommand{\spec}{\textrm{Spec}}
\newcommand{\dlog}{\textrm{dlog}}

\newcommand{\A}{{\mathbb A}}
\newcommand{\B}{{\mathbb B}}
\newcommand{\C}{{\mathbb C}}
\newcommand{\D}{{\mathbb D}}
\newcommand{\E}{{\mathbb E}}
\newcommand{\F}{{\mathbb F}}
\newcommand{\G}{{\mathbb G}}
\renewcommand{\H}{{\mathbb H}}
\newcommand{\I}{{\mathbb I}}\

\newcommand{\J}{{\mathbb J}}
\renewcommand{\L}{{\mathbb L}}
\newcommand{\M}{{\mathbb M}}
\newcommand{\N}{{\mathbb N}}
\renewcommand{\P}{{\mathbb P}}
\newcommand{\Q}{{\mathbb Q}}
\newcommand{\Qbar}{\overline{\Q}}
\newcommand{\R}{{\mathbb R}}
\newcommand{\SSS}{{\mathbb S}}
\newcommand{\T}{{\mathbb T}}
\newcommand{\U}{{\mathbb U}}
\newcommand{\V}{{\mathbb V}}
\newcommand{\W}{{\mathbb W}}
\newcommand{\Z}{{\mathbb Z}}
\newcommand{\Spec}{{\rm Spec}}

\thanks{The first named author is supported by National Natural Science Foundation of China (Grant No. 11622109, No. 11721101)}
\maketitle
\begin{abstract}
We study the $E_1$-degeneration of the logarithmic Hodge to de Rham spectral sequence of the special fiber of a semistable family over a discrete valuation ring. On the one hand, we prove that the $E_1$-degeneration property is invariant under admissible blow-ups. Assuming functorial resolution of singularities over $\mathbb{Z}$, this implies that the $E_1$-degeneration property depends only on the generic fiber. On the other hand, we show by explicit examples that the decomposability of the logarithmic de Rham complex is not invariant under admissible blow-ups, which answer negatively an open problem of L. Illusie (Problem 7.14 \cite{Illusie2002}). We also give an algebraic proof of an $E_1$-degeneration result in characteristic zero due to Steenbrink and Kawamata-Namikawa.
\end{abstract}

\section{Introduction}
Let $R$ be a discrete valuation ring with perfect residue field $k$, set $S=\Spec(R)$, and let $X\to S$ be a proper flat scheme over $R$ which is a semistable family (Definition \ref{defn_ssreduction}).  Logarithmic geometry in the sense of Fontaine-Illusie-Kato upgrades the structural morphism $X\to S$ naturally to a \emph{log smooth} morphism between log schemes \cite[Example 3.7 (2)]{KKato1988}. Restricting to the closed point of $S$, one obtains a smooth morphism ${\bf X_0}\to {\bf k}$ of log schemes. The following is a basic problem in the Hodge theory of semistable families:
\begin{prob}\label{$E_1$-degeneration problem}
When does the logarithmic Hodge to de Rham spectral sequence 
$$
E_1^{i,j}=H^j(X_0, \Omega^i_{{\bf X_0}/{\bf k}})\Longrightarrow H^{i+j}_{dR}({\bf X_0}/{\bf k})
$$
degenerate at $E_1$, where $H^{i+j}_{dR}({\bf X_0}/{\bf k})$ is the hypercohomology of the relative logarithmic de Rham complex $\Omega^{\bullet}_{{\bf X_0}/{\bf k}}$ associated to the morphism ${\bf X_0}\to {\bf k}$ of log schemes?
\end{prob}  
Classical Hodge theory gives the affirmative answer to the simplest case of the above problem, namely the case where $\mathrm{char}\ k=0$ and $X_0$ is smooth over $k$.  Keeping the assumption $\mathrm{char}\ k=0$, work of Steenbrink  \cite[Corollary 4.20 (ii)]{Steenbrink1976} implies
\begin{thm}[Steenbrink, $\mathrm{char}\ k=0$]\label{Steenbrink's thm}
Notation as above.  Then the above spectral sequence degenerates at $E_1$ when $X$ is a strictly semistable family over $S$.
\end{thm}
 While the existence of a smooth proper generic fiber degenerating to $X_0$ is crucial to Steenbrink's argument, it is unnecessary for the truth of Problem \ref{$E_1$-degeneration problem} when $\mathrm{char}\ k=0$. Recall that a normal crossing variety $X$ over $k$ is \emph{$d$-semistable} if the infinitesimal normal bundle $\sT^1_X$ is isomorphic to $\sO_{X_{\rm sing}}$, where $X_{\rm sing}$ is the singular locus of $X$.  We shall give an algebraic proof of the following result, which is essentially due to Kawamata-Namikawa \cite[Lemma 4.1]{KN1994}.

\begin{prop*}[Kawamata-Namikawa, Proposition \ref{reproof on KN}]\label{improvement of Steenbrink}
For a proper $d$-semistable normal crossing variety over a field $k$ of characteristic 0, the Hodge to de Rham spectral sequence in Problem \ref{$E_1$-degeneration problem} degenerates at $E_1$. 
\end{prop*}
\begin{rmk}
Work of Friedman \cite{Friedman1983} and Persson-Pinkham \cite{PP1983} imply that Proposition \ref{reproof on KN} is strictly more general than Theorem \ref{Steenbrink's thm}.
\end{rmk}
Next, we consider the case where $R$ is the ring of integers of a complete valued field $K$ of mixed characteristic. The fundamental work of Deligne-Illusie \cite{Del_Ill1987} provides us the affirmative answer of Problem \ref{$E_1$-degeneration problem}  for the case that $R$ is the Witt ring $W(k)$ and $X$ is proper smooth over $R$ of relative dimension $\leq \mathrm{char}\ k=p$. Their key result is the de Rham decomposition theorem: for a \emph{smooth} variety $X_0$ over $k$, it is $W_2(k)$-liftable if and only if $\tau_{<p}F_{\ast}\Omega^{\bullet}_{X_0/k}$ is isomorphic to $\bigoplus_{i=0}^{p-1}\Omega^i_{X'_0/k}[-i]$ in the derived category $D(X_0')$, where $X'_0=X_0\times_{k,\sigma}X_0$ (with $\sigma$ the Frobenius automorphism of $k$) and $F: X_0\to X'_0$ is the relative Frobenius. Now for a semistable reduction $X$ over $R=W(k)$,  L. Illusie raised the following problem:
\begin{prob}[Illusie,  Problem 7.14 \cite{Illusie2002}]\label{prob_Illusie}
Is the truncated logarithmic de Rham complex $\tau_{<p}F_{\ast}\Omega^{\bullet}_{{\bf X_0/k}}$
decomposable in $D(X'_0)$?
\end{prob}
In below, the log scheme ${\bf X_0}$ is said to be \emph{de Rham-decomposable} (DR-decomposable) if the answer to Problem \ref{prob_Illusie} is affirmative for ${\bf X_0}$. Clearly, Problem \ref{$E_1$-degeneration problem} is affirmative for a semistable family over $W(k)$ of relative dimension $<p$ whose special fiber $\bf X_0$ (as log scheme over $\bf k$) is DR-decomposable. The Illusie's problem for the curve case is affirmative for cohomological reason (Remark \ref{rmk_vanish_ob} (1)). Beyond this case, our answer to it is surprisingly NO.  See \S4 for counterexamples of arbitrarily large dimension.

The construction of our examples belongs to one of the simplest kinds of operations in algebraic geometry: blow-ups.  In the setting of semistable families, we introduce the following 
\begin{defn}\label{defn_admissible_blow-up}
Let $Z$ a semistable family over $S$. A blow-up of $Z$ with the center $Y_0\subset Z$ is said to be \emph{admissible} if $Y_0$ is contained in $Z_0$, regular and has normal crossings with $Z_0$ (Definition \ref{defn_NC_subscheme}).
\end{defn}
It is clear that an admissible blow-up of a semistable family over $S$ is again a semistable family over $S$. We prove that the $E_1$-degeneration property is invariant under admissible blow-ups.

\begin{thm*}[Theorem \ref{thm_degeneration_E1}]\label{thm A}
Let $R$ be a discrete valuation ring and $Z$ a semistable family over $S$. Let $X$ be an admissible blow-up of $Z$. Then Problem \ref{$E_1$-degeneration problem} holds true for ${\bf Z_0}$ if and only if it holds true for ${\bf X_0}$. 
\end{thm*}
We can make a step further by assuming a conjecture in resolution of singularities.  
\begin{cor*}[Corollary \ref{cor_degeneration_E1}]
Notation as above. Assume that the functorial embedded resolution of singularities applies over $\mathbb{Z}$\cite[\S 2.2.6]{Abr2016}. Then the $E_1$-degeneration property of the special fiber of $X/S$ depends only on its generic fiber $X_{\eta}$. Namely, for another semistable integral model $Y$ over $S$ of $X_\eta$, $E_1$-degeneration for $\bf X_0/k$ holds if and only if it holds  for $\bf Y_0/k$.
\end{cor*} 

On the other hand, the behavior of Problem \ref{prob_Illusie} under admissible blow-ups is more subtle.
\begin{thm*}[Corollary \ref{cor_decompose_blow-up} and \S4, Example 1, 2]
	Let $R$ be a discrete valuation ring and $Z$ a semistable family over $S$. Let $X$ be an admissible blow-up of $Z$. The de Rham decomposability of ${\bf X_0}$ implies that of ${\bf Z_0}$, but not vise-versa.
\end{thm*}
Therefore, our examples show that, even when the dimension of $X_0$ is strictly less than the characteristic of $k$, DR-decomposability is definitely \emph{stronger} than $E_1$ degeneration of the Hodge to de Rham spectral sequence, a fact that seems to have been previously unknown.

The counterexamples of W. Lang \cite{Lang1995} to Problem \ref{$E_1$-degeneration problem} with $K/\Q_p$ wildly ramified and those constructed in \S4 to Prolem \ref{prob_Illusie} with $K/\Q_p$ unramified motivate us to introduce a stronger lifting condition on log schemes, which shall guarantee the truth of Problem \ref{prob_Illusie} over a general $p$-adic base. Now let $R$ be the ring of integers of a $p$-adic local field. Equipping $S$ with the log structure given by the maximal ideal of $R$ and $\spec(W(k)[[t]])$ with the log structure given by the divisor $(t)$, one finds that $S$ is naturally a closed log subscheme of  $\spec(W(k)[[t]])$ by mapping $t$ to a uniformizer of $R$ (see \S3). 

\begin{prop*}[Proposition \ref{sufficient condition on illusie's problem}]\label{thm B}
Notation as above. Let $X$ be a semistable family over $R$. If $X$ is liftable to the $\spec(W(k)[[t]])$ as log schemes, then Problem \ref{prob_Illusie} holds true for $\bf X_0/k$. Consequently, if $X$ has relative dimension $<p$, then Problem \ref{$E_1$-degeneration problem} holds for ${\bf X_0/k}$.  
\end{prop*}

The paper is organized as follows: In Section \S2, we prove the $E_1$-degeneration theroem (Theorem \ref{thm_degeneration_E1}) for admissible blow-ups; In Section \S3, we provide a criterion for the DR-decomposability (Theorem \ref{thm_decom_lifting}) and then a sufficient condition for DR-decomposability (Proposition \ref{sufficient condition on illusie's problem}). We deduce Proposition \ref{reproof on KN} from the criterion by mod $p$ reduction.  In Appendix, we include some preliminaries on semistable reductions which are mainly used in \S2. 

\section{An $E_1$-degeneration theorem for admissible blow-ups}
It is $E_1$-degeneration of the Hodge to de Rham spectral sequence which guarantees Kodaira's vanishing theorem \cite{EV1986}. In this section, we aim to establish an $E_1$-degeneration property for semistable families. 

Let $R$ be a DVR with the residue field $k$ and $Z$ a semistable family over $S=\spec(R)$. Let $X/S$ be an admissible blow-up of $Z$ with the blow-up center $Y_0\subset Z_0$ (Definition \ref{defn_admissible_blow-up}). Let $\bf Z_0\to k$ and $\bf X_0\to k$ be the naturally associated log schemes. 
\begin{thm}\label{thm_degeneration_E1}
Notation as above. The following statements are equivalent:
	\begin{enumerate}
		\item the logarithmic Hodge to de Rham spectral sequence
		$$E_1^{ij}=H^j(Z_0,\Omega^i_{{\bf Z_0}/{\bf k}})\Rightarrow \H^{i+j}(Z_0,\Omega_{{\bf Z_0}/{\bf k}}^{\bullet})$$
		degenerates at $E_1$.
		\item the logarithmic Hodge to de Rham spectral sequence
		$$E_1^{ij}=H^j(X_0,\Omega^i_{{\bf X_0}/{\bf k}})\Rightarrow \H^{i+j}(X_0,\Omega_{{\bf X_0}/{\bf k}}^{\bullet})$$
		degenerates at $E_1$.
	\end{enumerate} 
\end{thm}
By the flat base change property of cohomologies, it suffices to assume that $k$ is algebraically closed. We shall use this assumption throughout this section.

The blow-up induces the morphism  $\pi:{\bf X_0}\rightarrow {\bf Z_0}$ of log schemes over $\bf k$. It induces the following morphism for each $i\geq 0$:
$$\pi^{\ast i}:\Omega^i_{{\bf Z_0}/{\bf k}}\rightarrow R\pi_\ast\Omega^i_{{\bf X_0}/{\bf k}}.$$
Our main technical result in the proof of Theorem \ref{thm_degeneration_E1} is the following


\begin{prop}\label{prop_KEY}
For each $i$, $$\pi^{\ast i}:\Omega^i_{{\bf Z_0}/{\bf k}}\rightarrow R\pi_\ast\Omega^i_{{\bf X_0}/{\bf k}}.$$
is an isomorphism in the derived category $D(Z_0)$.
\end{prop}

\begin{proof}[Proof of Theorem \ref{thm_degeneration_E1}]
By Proposition \ref{prop_KEY}, we have the following consequences
	\begin{enumerate}
		\item the natural morphism
		$$\Omega^{\geq l}_{{\bf Z_0}/{\bf k}}\to \pi_\ast\Omega^{\geq l}_{{\bf X_0}/{\bf k}}$$
		is an isomorphism of complexes, for each $l\geq 0$. 
		\item the natural morphism
		$$\pi_\ast\Omega^{\geq l}_{{\bf X_0}/{\bf k}}\to R\pi_\ast\Omega^{\geq l}_{{\bf X_0}/{\bf k}}$$
		is an isomorphism in $D(Z_0)$, for each $l\geq 0$. 
	\end{enumerate}
	This proves that for each $l\geq 0$, the vertical morphisms in the diagram
	\begin{align}
	\xymatrix{
		\mathbb{H}^i(Z_0,\Omega^{\geq l}_{{\bf Z_0}/{\bf k}})\ar[r]\ar[d]^{\pi_k^\ast}& \mathbb{H}^{i}(Z_0,\Omega_{{\bf Z_0}/{\bf k}}^{\bullet})\ar[d]^{\pi^\ast}\\
		\mathbb{H}^i(X_0,\Omega^{\geq l}_{{\bf X_0}/{\bf k}})\ar[r]& \mathbb{H}^{i}(X_0,\Omega_{{\bf X_0}/{\bf k}}^{\bullet})
	}
	\end{align}
	are isomorphisms. This proves the theorem.
\end{proof}

The following Lemma is well known but the authors cannot find a suitable reference. We give a proof for the convenience of the readers.
\begin{lem}\label{lem_blow-up}
	Let $Y$ be a smooth variety and $Z\subset Y$ be the smooth closed subvariety. Let $p:Y':={\rm Bl}_{Z}Y\to Y$ be the blowing up. Then the canonical morphism
	$$\sO_Y\simeq Rp_\ast\sO_{Y'}$$
	is a quasi-isomorphism.
\end{lem}
\begin{proof}
	Since the statement is \'etale locally, we may assume that $Y=\spec(k[x_1,\dots,x_n])$ and $Z$ is defined by the ideal $(x_1,\dots,x_r)$. Then $Y'$ is isomorphic to the subvariety of $Y\times \mathbb{P}^{r-1}$ defined by
	$$x_iy_j=x_jy_i,\quad i,j=1,\dots,r$$
	under the coordinates $(x_1,\dots,x_n,[y_1,\cdots,y_r])$.
	
	Zariski's Main theorem \cite[Corollary 11.4]{Hartshorne} gives us
	$$\sO_Y\simeq p_\ast\sO_{Y'}.$$
	Therefore it suffices to show that
	$$H^i(Y',\sO_{Y'})=0,\quad i>0.$$
	Recall that the composition morphism
	$$Y'\subset Y\times \mathbb{P}^{r-1}\to \mathbb{P}^{r-1}$$
	makes $Y'$ be a line bundle over $\mathbb{P}^{r-1}$, which by local calculation is isomorphic to $\sO(1)$. Therefore
	$$H^i(Y',\sO_{Y'})\simeq H^i(\mathbb{P}^{r-1},{\rm Sym}\sO(1))=0,\quad i>0$$
	and we proved the lemma.
\end{proof}
\begin{proof}[Proof of Proposition \ref{prop_KEY}]
We first remark that for a semistable reduction $Z/S$, we have a natural isomorphism
$$\Omega_{{\bf Z_0}/{\bf k}}\simeq\Omega_{Z/S}(\log Z_0)|_{Z_0},$$
where $\Omega_{Z/S}(\log Z_0)$ is the relative log cotangent sheaf attached to the log morphism $(Z,Z_0)\to(S,0=\spec(k))$.

Let $D\nsubseteq Z_0$ be a strict normal crossing divisor in $Z$, such that $D\cup Z_0$ is normal crossing and has normal crossings with $Y_0$ (Definition \ref{defn_NC_subscheme}). Denote $\tilde{D}$ for the strict transform on $X$. We will show, by induction on $\dim Z$ and the number of components of $D$, that the canonical morphism
\begin{align}\label{align_quasiiso}
\Omega^i_{Z/S}({\rm log} (Z_0+D))|_{Z_0}\to R\pi_\ast\left(\Omega^i_{X/S}({\rm log} (X_0+\tilde{D}))|_{X_0}\right)
\end{align}
is a quasi-isomorphism. That is, 
\begin{align}\label{align_quasiiso1}
\Omega^i_{Z/S}({\rm log} (Z_0+D))|_{Z_0}\stackrel{\simeq}{\to}\pi_\ast\left(\Omega^i_{X/S}({\rm log} (X_0+\tilde{D}))|_{X_0}\right)
\end{align}
and
\begin{align}\label{align_quasiiso2}
R^j\pi_\ast\left(\Omega^i_{X/S}({\rm log} (X_0+\tilde{D}))|_{X_0}\right)=0,\quad j>0.
\end{align}
Notice that when $D=\emptyset$, these isomorphisms are nothing but the statements of Proposition \ref{prop_KEY}.

Since these isomorphisms can be checked \'etale locally, we assume, in the rest of the proof, that $Z$ is affine and is a strictly semistable reduction over $R$. The proof is divided into two cases.

{\bf Case I: $i=0$.} It suffices to show
\begin{align}\label{align_XZ}
\sO_{Z_0}\simeq R\pi_\ast\sO_{X_0}.
\end{align}
Denote the irreducible decomposition by $Z_0=\cup_{i=1}^r Z_i$, then $X_0=\bigcup_{i=1}^{r+1} X_i$.
The restriction morphism $\pi_i:X_i\to Z_i$ is the blowing up along $Y_0\cap Z_i$. $X_{r+1}=\mathbb{P}(N_{Y_0/Z})$ and the restriction morphism $\pi_{r+1}:\mathbb{P}(N_{Y_0/Z})\to Y_0$ is the canonical projection.

For convenience, denote $Z_{r+1}=Y_0$. Given $I\subset\{1,\dots,r+1\}$, denote 
$$Z_I=\bigcap_{i\in I}Z_i.$$ 
The same notation applies to $X_0$. If $r+1\notin I$, the restriction morphism $\pi_I:X_I\to Z_I$ is the blowing up morphism along the smooth center $Z_I\cap Y_0$. If $r+1\in I$, then $\pi_I:X_I\to Z_I$ is a projective bundle.

We have the exact sequences
\begin{align}\label{align_X_0}
0\to\sO_{X_0}\to\bigoplus_{|I|=1}\sO_{X_I}\to\bigoplus_{|I|=2}\sO_{X_I}\to\cdots,
\end{align}
and
\begin{align}\label{align_Z_0}
0\to\sO_{Z_0}\to\bigoplus_{|I|=1}\sO_{Z_I}\to\bigoplus_{|I|=2}\sO_{X_I}\to\cdots.
\end{align}
Applying $R\pi_\ast$ to (\ref{align_X_0}) we obtain the spectral sequence
$$E_1^{ij}=\bigoplus_{|I|=i}R^j\pi_{I\ast}\sO_{X_I}\Rightarrow R^{i+j}\pi_\ast\sO_{X_0}.$$
By Lemma \ref{lem_blow-up} we have
$$E_1^{ij}\simeq\begin{cases}
\bigoplus_{\stackrel{}{|I|=i}}\sO_{Z_I}, & j=0 \\
0, & j\neq 0
\end{cases}.$$
Combining with the exactness of (\ref{align_Z_0}), we obtain that (\ref{align_XZ}) is an quasi-isomorphism.
 
{\bf Case II: $i>0$.}
For convenience, we denote the claimed quasi-isomorphism (\ref{align_quasiiso}) by $\mathbf{C}_{n,r}$, where $n=\dim Z$ and $r=l(D)$ is the number of the components of $D$. Certainly the assertion $\mathbf{C}_{n,r}$ depends on the geometric setting $(Z/W,Y_0)$. We abuse this notation when $(Z/W,Y_0)$ is clear in the context.

Denote $l_{Y_0}(Z_0)$ for the number of components of $Z_0$ that contain $Y_0$. Since $D\cup Z_0$ is strict normal crossing, we have
\begin{align}\label{align_r}
0\leq r\leq\dim Z-l_{Y_0}(Z_0)-\dim Y_0
\end{align}
in the assertion $\mathbf{C}_{n,r}$. When $r$ reach its maximum, i.e. $r=\dim Z-l_{Y_0}(Z_0)-\dim Y_0$, the defining functions of $D$ and of the components of $Z_0$ that contain $Y_0$ form a local frame of the conormal bundle $\sI_{Y_0,Z}/\sI_{Y_0,Z}^2$. 


The assertion $\mathbf{C}_{1,0}$ is trivial. The remaining proof is divided into two inductive steps:
 
\textbf{Step 1:} In this step we prove $\mathbf{C}_{n,r(n)}$ for each $n>0$. Here 
$$r(n)=\dim Z-l_{Y_0}(Z_0)-\dim Y_0$$ 
is the maximum of $r$ in (\ref{align_r}) under the geometric setting $(Z/W,Y_0)$. 


We are going to show that the canonical morphism
\begin{align}\label{align_step1_iso}
\pi^\ast\Omega_{Z/S}({\rm log}(Z_0+D))\to\Omega_{X/S}({\rm log}(X_0+\tilde{D}))
\end{align}
is an isomorphism.  As a consequence, for each $i$ the projection formula reads
$$R\pi_\ast\left(\Omega^i_{X/S}({\rm log}(X_0+\tilde{D}))|_{X_0}\right)\simeq R\pi_\ast\pi^\ast\left(\Omega^i_{Z/S}({\rm log}(Z_0+D))|_{Z_0}\right)\simeq\Omega^i_{Z/S}({\rm log}(Z_0+D))|_{Z_0}\otimes R\pi_\ast\sO_{X_0}.$$
By case 1, $R\pi_\ast\sO_{X_0}\simeq\sO_{Z_0}$. This finishes the proof of step 1.

It remains to prove the isomorphism (\ref{align_step1_iso}). Since (\ref{align_step1_iso}) is generically an isomorphism between locally free sheaves, it is injective. Without loss of generality, we assume that 
\begin{enumerate}
	\item $Z=\spec(A)$ for a Noetherian regular local ring $A$.
	\item $Y_0$ is defined by the ideal $(x_1,\dots,x_m)$ where $x_1,\dots,x_m\in A$ are the defining sections of the components of $Z_0\cup D$ that contains $Y_0$.
	\item $X=\spec\left(A[y_1,\dots,y_{m-1}]/(x_my_1-x_1,\dots,x_m y_{m-1}-x_{m-1})\right)$.
\end{enumerate}
Denote ${\bf X}=(X,X_0\cup \tilde{D})$ and ${\bf Z}=(Z,Z_0\cup D)$. Then the cokernel of (\ref{align_step1_iso}) is given by
\begin{align}\nonumber
\Omega_{{\bf X}/{\bf Z}}&\simeq\bigoplus_{i=1}^{m-1}\sO_X\frac{dy_i}{y_i}\oplus\sO_X\frac{dx_m}{x_m}/\bigoplus_{i=1}^{m-1}\sO_X(\pi^\ast(\frac{dx_i}{x_i}))\oplus\sO_X\frac{dx_m}{x_m}\\\nonumber
&\simeq\bigoplus_{i=1}^{m-1}\sO_X\frac{dy_i}{y_i}\oplus\sO_X\frac{dx_m}{x_m}/\bigoplus_{i=1}^{m-1}\sO_X(\frac{dy_i}{y_i}+\frac{dx_m}{x_m})\oplus\sO_X\frac{dx_m}{x_m}=0.\nonumber
\end{align}
Hence (\ref{align_step1_iso}) is an isomorphism.

\textbf{Step 2:}
Assume that $r<\dim Z-l_{Y_0}(Z_0)-\dim Y_0$. By Lemma \ref{lem_induction_ssreduction} there is a regular divisor $D'\nsubseteq Z_0$ on $Z$ such that $D'\cup D\cup  Z_0$ is strict normal crossing. Denote by $\tilde{D'}$ the strict transform on $X$ and by $\tilde{D'}_0$ its central fiber. Then $D'$ and $\tilde{D'}$ are semistable reductions over $S$. By Lemma \ref{lem_residue_sequence} we have the following exact sequence

\begin{align}\label{align_residue2}
0\to\Omega^i_{X/S}({\rm log} (X_0+\tilde{D}))|_{X_0}\to\Omega^i_{X/S}({\rm log} (X_0+\tilde{D}+\tilde{D}'))|_{X_0}\stackrel{{\rm Res}_{\tilde{D'}}|_{X_0}}{\to}\Omega^{i-1}_{\tilde{D'}/S}({\rm log} (X_0+\tilde{D})|_{\tilde{D'}})|_{\tilde{D'}_0}\to 0.
\end{align}
If $\mathbf{C}_{\dim Z,l(D)+1}$ and $\mathbf{C}_{\dim Z-1,\ast}$ hold true, we have the exact sequence
\begin{align}\label{align_induction}
0&\to\pi_\ast\left(\Omega^i_{X/S}({\rm log}(X_0+\tilde{D}) )|_{X_0}\right)\to\pi_\ast\left(\Omega^i_{X/S}({\rm log} (X_0+\tilde{D}+\tilde{D'}))|_{X_0}\right)\\
&\stackrel{\pi_\ast({\rm Res}_{\tilde{D'}}|_{X_0})}{\to}\pi_\ast\left(\Omega^{i-1}_{\tilde{D'}/S}({\rm log} ((X_0+\tilde{D})|_{\tilde{D'}}))|_{\tilde{D'}_0}\right)\to R^1\pi_\ast\left(\Omega^i_{X/S}({\rm log} (X_0+\tilde{D}))|_{X_0}\right)\to 0,\nonumber
\end{align}

and 
$$R^{j+1}\pi_\ast\left(\Omega^i_{X/S}({\rm log}(X_0+\tilde{D}))|_{X_0}\right)\simeq
R^j\pi_\ast\left(\Omega^{i-1}_{\tilde{D}/S}({\rm log} ((X_0+\tilde{D})|_{\tilde{D'}}))|_{\tilde{D'}_0}\right)=0 ,\quad j>0.
$$
By $\mathbf{C}_{\dim Z,l(D)+1}$ and $\mathbf{C}_{\dim Z-1,\ast}$, we have
$$
\pi_\ast\left(\Omega^i_{X/S}({\rm log} (X_0+\tilde{D}+\tilde{D'}))|_{X_0}\right)\simeq
\Omega^{i}_{Z/S}({\rm log} (Z_0+D+D'))|_{Z_0},
$$
and
$$
\pi_\ast\left(\Omega^{i-1}_{\tilde{D'}/S}({\rm log} ((X_0+\tilde{D})|_{\tilde{D'}}))|_{\tilde{D'}_0}\right)\simeq
\Omega^{i-1}_{D'/S}({\rm log} (Z_0+D))|_{D'_0}.
$$
Therefore (\ref{align_induction}) is isomorphic to 
\begin{align*}
0&\to\pi_\ast\left(\Omega^i_{X/S}({\rm log} (X_0+\tilde{D}))|_{X_0}\right)\to\Omega^i_{Z/S}({\rm log} (Z_0+D+D'))|_{Z_0}\\
&\stackrel{{\rm Res}_{D'}|_{D'_0}}{\to}\Omega^{i-1}_{D'/S}({\rm log} ((Z_0+D)|_{D'}))|_{D'_0}\to R^1\pi_\ast\left(\Omega^i_{X/S}({\rm log}(X_0+\tilde{D}))|_{X_0}\right)\to 0.\nonumber
\end{align*}
As a consequence, 
$$\pi_\ast\left(\Omega^i_{X/S}({\rm log} (X_0+\tilde{D}))|_{X_0}\right)\simeq \Omega^i_{Z/S}({\rm log} (Z_0+D))|_{Z_0}$$
and
$$R^1\pi_\ast\left(\Omega^i_{X/S}({\rm log} (X_0+\tilde{D}))|_{X_0}\right)=0.$$
In summary, we obtain that 
$$\mathbf{C}_{n,r+1} \textrm{ and } \mathbf{C}_{n-1,r} \textrm{ imply }\mathbf{C}_{n,r}.$$
Together with step 1, the proposition is proved.
\end{proof}
To make a step further,  we make the following conjecture.
\begin{conj}\label{conj}
	Let $X$ and $Y$ be two semistable families over $S$ with an isomorphism $\alpha:X_{\eta}\simeq Y_{\eta}$ of generic fibers. Then, there is a sequence of rational maps
	$$\varphi:X=V_0\stackrel{\varphi_1}{\dashrightarrow}V_1\stackrel{\varphi_2}{\dashrightarrow}V_2\stackrel{\varphi_3}{\dashrightarrow}\cdots\stackrel{\varphi_n}{\dashrightarrow}V_n=Y,$$
	such that
	\begin{enumerate}
		\item $\varphi|_{\eta}=\alpha$,
		\item For each $i$, $V_i$ is a semi-stable reduction over $S$, 
		\item For each $i$, either $\varphi_{i}$ or $\varphi^{-1}_{i}$ is an admissible blow-up.
	\end{enumerate}
\end{conj}
The conjecture follows if we assume that the functorial embedded resolution of singularities applies over $\mathbb{Z}$ (\cite[Theorem 1.3.3 (2)]{Abr2016}). In particular, it holds when ${\rm char}(k)=0$ (\cite[Theorem 0.3.1]{Abr2002}). As a consequence
\begin{cor}\label{cor_degeneration_E1}
	Notation as above and assume that the functorial embedded resolution of singularities applies over $\mathbb{Z}$. Then the $E_1$-degeneration property of the special fiber of $X/S$ depends only on the generic fiber $X_{\eta}$. Namely, for another semistable integral model $Y$ over $S$ of  $X_\eta$, $E_1$-degeneration for $\bf X_0/k$ holds if and only if it holds  for $\bf Y_0/k$.
\end{cor}

To conclude this section, we shall deduce an important consequence from Proposition \ref{prop_KEY} on DR-decomposability, a topic on which we shall concentrate in the next two sections. Together with Examples constructed in \S4, the following result forms the corresonding picture on DR-decomposability under admissible blow-ups.
\begin{cor}\label{cor_decompose_blow-up}
Notations as above. Assume moreover that ${\rm char}(k)=p>0$. Denote by $F_{X_0}:X_0\to X_0$ and $F_{Z_0}:Z_0\to Z_0$ the Frobenius morphism. If $\tau_{<p}F_{X_0\ast}\Omega_{\bf X_0/k}^\bullet$ is decomposable in $D(X_0)$, then $\tau_{<p}F_{Z_0\ast}\Omega_{\bf Z_0/k}^\bullet$ is decomposable in $D(Z_0)$.
\end{cor}
\begin{proof}
By the assumption we have an isomorphism
\begin{align}\label{align_X_0_decompose}
\tau_{<p}F_{X_0\ast}\Omega_{\bf X_0/k}^\bullet\simeq\bigoplus_{i=0}^{p-1}\Omega_{\bf X_0/k}^i[-i]
\end{align}
in $D(X_0)$. Applying $R\pi_\ast$ on (\ref{align_X_0_decompose}) we get
\begin{align}\label{align_Z_0_decompose}
R\pi_\ast\tau_{<p}F_{X_0\ast}\Omega_{\bf X_0/k}^\bullet\simeq\bigoplus_{i=0}^{p-1}R\pi_\ast\Omega_{\bf X_0/k}^i[-i].
\end{align}
By Proposition \ref{prop_KEY}, the right hand side of (\ref{align_Z_0_decompose}) is canonically isomorphic to 
$$\bigoplus_{i=0}^{p-1}\Omega_{\bf Z_0/k}^i[-i].$$
Therefore it remains to show that the canonical morphism
\begin{align}\label{align_Z_0_decompose1}
\tau_{<p}F_{Z_0\ast}\Omega_{\bf Z_0/k}^\bullet\to R\pi_\ast\tau_{<p}F_{X_0\ast}\Omega_{\bf X_0/k}^\bullet
\end{align}
is an isomorphism in $D(Z_0)$. By \cite[Theorem 4.12]{KKato1988} the spectral sequence associated to the canonical truncation of $\tau_{<p}F_{X_0\ast}\Omega^\bullet_{{\bf X_0}/{\bf k}}$ is
$$E_1^{ij}=\begin{cases}
R^j\pi_\ast\Omega^i_{{\bf X_0}/{\bf k}}, & i<p \\
0, & i\geq p
\end{cases}\Longrightarrow R^{i+j}\pi_\ast \tau_{< p}F_{X_0\ast}\Omega^\bullet_{{\bf X_0}/{\bf k}}.$$
By Proposition \ref{prop_KEY} we obtain that
$$H^i(R\pi_\ast \tau_{<p}F_{\ast}\Omega^\bullet_{{\bf X_0}/{\bf k}})\simeq\begin{cases}
\Omega^i_{{\bf Z_0}/{\bf k}}, & i<p \\
0, & i\geq p
\end{cases}.$$
Again by \cite[Theorem 4.12]{KKato1988}, (\ref{align_Z_0_decompose1}) is a quasi-isomorphism and we complete the proof.
\end{proof}

\section{A criterion for DR-decomposability}\label{section_log}
The decomposition theorem of Deligne-Illusie has a log analogue, that is the Kato's decomposition theorem (see  \cite[Theorem 4.12]{KKato1988}). Kato's theorem is crucial to the construction of our examples in the next section. Indeed, our construction is based on a simple modification (Theorem \ref{thm_decom_lifting}) of Kato's theorem. In order to understand it properly, let us start with recalling the basic notion in log geometry, namely \emph{log structure}, as introduced in \cite{KKato1988}.
\begin{defn}
Let $X$ be a scheme (equipped with the \'etale topology). A pre-log structure on $X$ is a homomorphism $\alpha: M\to \sO_X$ of sheaves of monoids, where $\sO_X$ is regarded as a sheaf of monoids with respect to the ring multiplication. A pre-log structure $(M,\alpha)$ is called a log structure if $\alpha$ induces an isomorphism
$$
\alpha^{-1}(\sO_X^*)\cong \sO_X^*,
$$
where $\sO_X^*\subset \sO_X$ is the subsheaf of invertible elements. A log scheme is a triple $(X,M,\alpha)$, where $X$ is a scheme equipped with a log structure $(M,\alpha)$ on $X$. When the meaning of $\alpha$ is clear, we shall omit it from the definition of a log scheme.
\end{defn}
To a pre-log structure $(M,\alpha)$, one associates canonically a log structure $M^a$ \cite[\S 1.3]{KKato1988}. For any morphism $P\to \Gamma(X,\sO_X)$ of monoids, we attach the pre-log structure (and hence the associated log structure): $P_X\to \sO_X$, where $P_X$ denotes for the constant sheaf on $X$ corresponding to $P$ \cite[\S 1.5 (3)]{KKato1988}. This type of log structure is important for us to elucidate different log structures on base schemes in our setting.  In our context, we shall \emph{only} apply $P=\N$, the additive monoid of nonnegative integers. Thus, in order to specify a morphism of the above type, it suffices to indicate an element of $\Gamma(X,\sO_X)$ which is the image of $1\in \N$. When there is no danger of confusion, a log scheme of the form $(X,1\mapsto a)$, for some $a\in \Gamma(X,\sO_X)$, shall mean a log structure on $X$ arising from this way.
\begin{exmp}
Let $k$ be an arbitrary field. The log scheme $(\spec(k), 1\mapsto 0)$ is called the \emph{standard log point} that is denoted by $\bf k$ throughout the paper. For $k$ a perfect field of characteristic $p>0$ and $n$ a natural number, $(\Spec(W_n(k)) ,1\mapsto p^m), 0\leq m\leq n$ are mutually non-isomorphic log schemes. For $m=0$, the log structure is trivial  \cite[\S1.5 (2)]{KKato1988}. For $n= m=1$, we get the standard log point $\bf k$. In the following, the log scheme ${\bf W_2}$ which is defined to be the case $n=2,m=2$ is especially important. Note both $\bf W_2$ and $(\spec(W_2), 1\mapsto p)$ are the first-order thickenings of $\bf k$. 
\end{exmp}
Let $R$ be a DVR with the residue field $k$ and $X$ be a semistable reduction over $R$.  Let $M_{X_0}$ be the log structure on $X$ attached to $X_0$ \cite[\S1.5 (1)]{KKato1988}. Then, the structural morphism $X\to \spec(R)$ extends to a \emph{log smooth} morphism of log schemes $f: (X,M_{X_0})\to (\spec(R),1\mapsto \pi)$ \cite[\S3]{KKato1988}, where $\pi$ is a uniformizer of $R$.  Let $f_0:{\bf X_0}:=(X_0,M_{X_0}|_{X_0})\to {\bf k}$ be the base change of $f$ via the exact closed immersion of log schemes ${\bf k}\hookrightarrow (\spec(R), 1\mapsto \pi)$. The most important property of $f_0$ is that it is of Cartier type (\cite[Definition 4.8]{KKato1988}). Thus Kato's decomposition theorem (\cite[Theorem 4.12]{KKato1988}) applies to $f_0$.

Now we restrict ourselve to the setting of Problem \ref{prob_Illusie}, that is $R=W(k)$ with $k$ a perfect field of positive characteristic $p$. Notice that we are already in the position to apply Kato's decomposition theorem to $f_0$ defined as above. If we take the lifting $(\Spec \ W_2(k), 1\mapsto p)$ of the base log scheme $\bf k$, it seems that we could conclude the decomposition property of $\tau_{<p}F_{*}\omega^{\bullet}_{{\bf X_0}/{\bf k}}$ by applying  \cite[Theorem 4.12 (2)]{KKato1988}.  However, this is not quite true: by a simple calculation, one finds the absolute Frobenius $F_{{\bf k}}: {\bf k}\to {\bf k}$ of the standard log point is \emph{non-liftable} to $(\Spec \ W_2(k), 1\mapsto p)$. Therefore, to the obvious lifting ${\bf X_1}$ of ${\bf X_0}$ over $(\Spec\ W_2, 1\mapsto p)$, which is given by
$$
{\bf X_1}:=(X,M_{X_0})\times_{(\Spec\ W, 1\mapsto p)}(\Spec\ W_2, 1\mapsto p),
$$
one cannot perform the log analogue of the base change as we do in the scheme case. As a consequence, the lifting condition for ${\bf X'_0}:={\bf X_0}\times_{{\bf k},F_{\bf k}}{\bf k}$, as required by the theorem, might not be satisfied. Our examples in the next section shall demonstrate that it is indeed the case.

The departure from trying to show the truth of Problem \ref{prob_Illusie} begins with the following simple observation. Recall that a log smooth variety ${\bf X}/{\bf k}$ is DR-decomposable if $\tau_{<p}F_{*}\omega^{\bullet}_{{\bf X}/{\bf k}}$ is decomposable in the derived category.

\begin{thm}\label{thm_decom_lifting}
Let $f: {\bf X}\to \bf k$ be a log smooth morphism of Cartier type. Then ${\bf X}$ is DR-decomposable if and only if ${\bf X}$ admits a log smooth lifting to ${\bf W}_2$.
\end{thm}
\begin{proof}
Let $f': {\bf X'}\to \bf k$ is the base change of $f$ by the absolute Frobenius $F_{\bf k}: \bf k\to \bf k$. Recall that ${\bf W}_2=(\spec(W_2(k)),1\mapsto0)$. There is an obvious lifting $F_{{\bf W}_2}: {\bf W}_2\to {\bf W}_2$ of $F_{\bf k}$ which is given by the following commutative diagram:
$$\xymatrix{
	W_2\ar[r]^{F_{W_2}} & W_2\\
	\mathbb{N} \ar[u]^\alpha \ar[r]^{\times p} & \mathbb{N}\ar[u]_\alpha,\\
}$$
where $F_{W_2}$ is the Frobenius morphism of $W_2$ and $\alpha$ is the pre-log structure determined by $1\mapsto 0$. Applying Kato's decomposition theorem (\cite[Theorem 4.12]{KKato1988}), one obtains the following statement: $\tau_{<p}F_{\ast}\Omega^{\bullet}_{\bf X/k}$ is decomposable in the derived category if and only if ${\bf X'}$ is liftable to ${\bf W}_2$. So it remains to show ${\bf X'}$ is liftable to ${\bf W}_2$ if and only if ${\bf X}$ is liftable to ${\bf W}_2$. One direction is clear: via the base change by $F_{{\bf W}_2}$, one obtains a $\bf{W}_2$-lifting of ${\bf X'}$ from that of ${\bf X}$. The converse direction is less obvious, as $F_{{\bf W}_2}$ is not an isomorphism of log schemes (for $F_{\bf k}$ is not an isomorphism). We prove it as follows: by \cite[Proposition 3.14]{KKato1988} and the proof therein, we have the obstruction class $\omega(f)\in {\rm Ext}^2(\Omega^1_{{\bf X}/{\bf k}}, f^*(p))$ (resp. $\omega(f')\in  {\rm Ext}^2(\Omega^1_{{\bf X'}/{\bf k}}, f'^*(p))$) of log smooth lifting of ${\bf X}$ (resp. ${\bf X'}$) over $\mathbf{W}_2$. We claim that $\omega(f)$ vanishes if and only if $\omega(f')$ does. For that, we take an open affine covering $\{{\bf U_i}\}_{i\in I}$ of ${\bf X}$ (each ${\bf U_i}$ is equipped with the induced log structure). Because of  \cite[Proposition 3.14 (1)]{KKato1988}, we may take for each $i\in I$ a log smooth lifting $\sU_i$ over $\mathbf{W}_2$, and over each overlap ${\bf U_{ij}}:={\bf U_i}\cap {\bf U_j}$ (which is again affine), we may take an isomorphism $\alpha_{ij}:\sU_i|_{{U_{ij}}}\rightarrow \sU_j|_{{U_{ij}}}$ between two liftings of $U_{ij}$ over ${\bf W}_2$. Then $\omega(f)$ is represented by the 2-cocycle $(c_{ijk})$ which over ${\bf U_{ijk}}:={\bf U_i}\cap {\bf U_j}\cap {\bf U_k}$ takes the value
$$
\alpha_{ki}^{-1}\alpha_{jk}\alpha_{ij}-Id\in {\rm Hom}(\Omega^1_{{\bf U_{ijk}}/{\bf k}},f^*(p))=T_{{\bf U_{ijk}}/{\bf k}}.
$$
Here $T_{?/{\bf k}}$ is the dual of $\Omega_{?/{\bf k}}$ for a log scheme $?$ over $\bf k$. The last equality uses the fact that $\Omega^1_{{\bf X}/{\bf k}}$ is locally free by  \cite[Proposition 3.10]{KKato1988}. Now we pull back the datum $\sU_i$s and $\{\alpha_{ij}\}$s over ${\bf W}_2$ by the morphism $F_{{\bf W}_2}$. A moment of thought shall lead us to the conclusion that $\omega(f')$ is represented by the 2-cocycle $(\sigma^*(c_{ijk}))$. In  other words, under the natural map
$$
H^2(X,T_{\bf X/\mathbf{k}})\stackrel{{\sigma^\ast}}{\longrightarrow} H^2(X',\sigma^\ast T_{\bf X/\mathbf{k}})=
H^2(X',T_{\bf X'/\mathbf{k}}),$$
the obstruction class of lifting $f$ is mapped to that of lifting $f'$. As $\sigma^*$ is semi-linear and bijective, $\omega(f)=0$ if and only if $\omega(f')=\sigma^*(\omega(f))=0$. Thus the claim is proved and then the theorem follows.
\end{proof}
\begin{rmk}\label{rmk_vanish_ob}
	By Theorem \ref{thm_decom_lifting} and \cite[Proposition 3.14]{KKato1988}, ${\bf X}$ is DR-decomposable whenever $H^2(X,T_{\bf X/k})=0$. This includes the cases that 
	\begin{enumerate}
		\item $\dim {\bf X}=1$,
		\item ${\bf X}$ is affine,
	\end{enumerate}
\end{rmk}
\begin{rmk}
	After presenting our results, Weizhe Zheng provided us a more conceptual proof of Theorem \ref{thm_decom_lifting}: Let ${\bf X}=(X,M_X)$, ${\bf X'}=(X',M_{X'})$. Denote by $\textrm{Lift}(X,M_X)$ (resp. $\textrm{Lift}(X',M_{X'})$) the groupoid of liftings of $(X,M_X)$ (rsep. $(X',M_{X'})$) over $\mathbf{W}_2$. Let $G:\mathbf{W}_2\rightarrow\mathbf{W}_2$ be a lifting of the log Frobenius morphism $F:\mathbf{k}\rightarrow\mathbf{k}$. Given a lifting $(X^{(1)},M_{X^{(1)}})\in \textrm{Lift}(X)$, the pullback of $(X^{(1)},M_{X^{(1)}})$ along $G$ gives an object in $\textrm{Lift}(X',M_{X'})$. With the obvious assignments on morphisms, one can get a functor
	$$
	A: \textrm{Lift}(X,M_X)\rightarrow\textrm{Lift}(X',M_{X'}).
	$$
	Conversely, let $(X'^{(1)},M_{X'^{(1)}})\in \textrm{Lift}(X',M_{X'})$ be a lifting of $(X',M_{X'})$. Denote by $i:(X',M_{X'})\hookrightarrow (X'^{(1)},M_{X'^{(1)}})$ the canonical strict closed immersion and by $\sigma:(X',M_{X'})\rightarrow (X,M_{X})$ the base change of $F:\mathbf{k}\rightarrow\mathbf{k}$. Recall that $\sigma:X'\rightarrow X$ is an isomorphism. One can construct the pushout $X'^{(1)}\amalg_{X'}X$ of the diagram
	$$\xymatrix{
		X' \ar[r]^{\sigma} \ar[d]^i & X\\
		X'^{(1)} &
	}$$
	as follows:
	\begin{itemize}
		\item The underlying scheme $X'^{(1)}\amalg_{X'}X$ is defined to be $X'^{(1)}$,
		\item the log structure on $X'^{(1)}\amalg_{X'}X$ is defined to be $M_{X'^{(1)}}\times_{M_{X'}}M_X$.
	\end{itemize}
	With the obvious assignments on morphisms, the pushout process along $\sigma:(X',M_{X'})\rightarrow (X,M_{X})$ gives a functor
	$$
	B: \textrm{Lift}(X',M_{X'})\rightarrow\textrm{Lift}(X,M_{X}).
	$$
	It is straightforward to check the following proposition.
	\begin{prop}
		The functor $A$ gives an  equivalence of groupoids, and the functor $B$ is its quasi-inverse.
	\end{prop}
\end{rmk}
In the following, we use Kato's decomposition theorem (or rather the version Theorem \ref{thm_decom_lifting} obtained as above), to prove Propositions \ref{reproof on KN} and Theorem \ref{sufficient condition on illusie's problem}.

Let $X$ a normal crossing (n.c.) variety defined over a field $K$. By  \cite[Theorem 5.4]{FKato1995}, $X$ is $d$-semistable if and only if $X$ admits a log structure of \emph{semistable type}. This fact was first shown by Kawamata-Namikawa \cite[Proposition 1.1]{KN1994} for simple normal crossing varieties. The related notion of log structure of semistable type is the notion of log structure of \emph{embedding type} (whose origin traces back to Steenbrink). Both notions are important to our argument below. However, for sake of brevity, we refer our readers to  \cite[\S4-5]{FKato1995} for precise definitions. Now we assume that $X$ is a $d$-semistable n.c. variety over $K$. Let $\alpha: M_X\to\sO_X$ be a log structure of semistable type on $X$. Let ${\bf K}=(\spec(K), 1\mapsto 0)$ be the standard log point. Then we have the following commutative diagram:

$$
\xymatrix{
	K\ar[r]^{ } &  \sO_X\\\
	\mathbb{N} \ar[u]^{ } \ar[r]^{\triangle} & M_X,\ar[u]_{\alpha }\\
}$$
which defines the morphism $f: \mathbf{X}:=(X,M_X)\to \bf K$ of log schemes (which is log smooth). Here $\alpha$ is etale locally isomorphic to the one associated to a SNCD in the affine space over $K$, and the morphism $\triangle$ on such a local chart is the diagonal map $\N\to \N^{r}$ with $r$ the number of local branches. The existence of a log structure of embedding type on $X$ does not necessarily offer such a commutative diagram. Also, we remark that over $X$ there could be more than one isomorphism classes of log structures of semistabe type. Using the description of  isomorphism classes of log structures in  \cite[Remark 4.7]{FKato1995}, it is easy to see that the sheaf $\Omega^1_{{\bf X}/\bf K}$ of relative log differential forms is independent of a choice of $M_X$. From now on, we denote this sheaf by $\Omega^1_{\bf X/K}$ and the corresponding de Rham complex by $\Omega^{\bullet}_{\bf X/K}$ as in Problem \ref{$E_1$-degeneration problem}. Kawamata-Namikawa \cite{KN1994} adapted the original transcendental method of Steenbrink \cite{Steenbrink1976} to establish the following result. Ours is to follow the char $p$ method of Deligne-Illusie \cite{Del_Ill1987} (see \cite[\S6]{Illusie2002}, especially the proof of Theorem 6.9 loc. cit.).
\begin{prop}\label{reproof on KN}
For a proper $d$-semistable n.c. variety over a field $K$ of characteristic 0, the Hodge to de Rham spectral sequence in Problem \ref{$E_1$-degeneration problem} degenerates at $E_1$. 
 \end{prop}
\begin{proof}
Write the field $K$ as an inductive limit of its sub $\Z$-algebras of finite type. Using  \cite[Lemma 6.1.2]{Illusie2002}, one finds a $\Z$-algebra $A\subset K$ of finite type, a proper normal crossing scheme $\tilde X$ of finite type over $S=\spec(A)$ and a log structure of semistable type 
$$
\tilde \alpha: \widetilde{M}_{\tilde X}\to \sO_{\tilde X}
$$ 
relative to $S$ on $\tilde X$ which yields the following commutative diagram:
$$
\xymatrix{
	A\ar[r]^{ } &  \sO_{\tilde X}\\\
	\mathbb{N} \ar[u]^{ } \ar[r]^{\tilde \triangle} & \widetilde{M}_{\tilde X},\ar[u]_{\tilde \alpha }\\
}$$
such that the associated log morphism $\tilde f: (\tilde X, \widetilde{M}_{\tilde X})\to (S, 1\mapsto 0):={\bf S}$ pulls back to $f$ via the obvious base change ${\bf K}\to {\bf S}$. That is, we `spread out' the log morphism $f$ to obtain an integral model $\tilde f: {\bf \tilde X}=(\tilde X,\widetilde{M}_{\tilde X})\to \bf S$. By  \cite[Proposition 6.3]{Illusie2002}, we may assume $S$ is smooth over $\Z$ by schrinking $S$ if necessary. Using the exact argument as  \cite[Proposition 6.6]{Illusie2002}, schrinking $S$ if necessary, we can assume that the $A$-modules $R^n\tilde f_*\Omega^\bullet_{\bf \tilde X/S}$ and $R^j\tilde f_*\Omega^i_{\bf \tilde X/S}$ are all free of finite rank and satisfies the cohomology base change for any morphism $(S', 1\mapsto 0)\to {\bf S}$. By Proposition 6.4 \cite{Illusie2002}, we may take a closed subscheme $s:\bf k\to {\bf S}$ with $\textrm{char}(k)$ sufficiently large (larger than the relative dimension of $\tilde f$). Since $(S,1\mapsto 0)$ is smooth over $(\spec(\Z), 1\mapsto 0)$, the morphism $s$ extends to a morphism $g: \bf W_2\to \bf S$. Pulling back $\tilde f$ via the base change $s$ and $g$, one obtains the following Cartesian diagrams:
$$
\begin{CD}
\bf Y @> >> \bf Y_1 @>>>\bf \tilde X@< <<\bf X\\
  @V\tilde f_0 VV @V \tilde f_1VV@V\tilde f VV@VVfV\\
{\bf k}@> >>{\bf W_2}@>g>> {\bf S}@<<<\bf K.
\end{CD}
$$
By the remark following  \cite[Definition 4.8]{KKato1988}, $\tilde f_0$ is of Cartier type.  Applying Theorem \ref{thm_decom_lifting} to the morphism $\tilde f_0$, one obtains the fact that $F_*\Omega^\bullet_{\bf Y/k}$ is decomposable in $D(Y')$ and therefore (follow the same arguments as in \cite[Corollary 5.6]{Illusie2002}) the Hodge to de Rham spectral sequence of the filtered complex $\Omega^\bullet_{\bf Y/k}$ degenerates at $E_1$. By the freeness and base change property stated as above, the Hodge to de Rham spectral sequence of the filtered complex $\Omega^\bullet_{\bf X/K}$ also degenerates at $E_1$.
\end{proof}
Now let $K$ be a $p$-adic local field and $V$ be its ring of integers. The most powerful tool to obtain the $E_1$-degeneration in Problem \ref{$E_1$-degeneration problem} for a semistable family over $V$ is to show the DR-decomposability in Problem \ref{prob_Illusie} (and the dimension is larger than the characteristic of the residue field $k$). Lang's example \cite{Lang1995} (smooth schemes over ramified $V$) and our example in \S4  (semistable families over unramified $V$) motivate us to find a common criterion for a general semistable reduction over $V$. To this purpose, for the special fiber $\bf X_0\to k$ we propose a stronger lifting condition than the original one in Deligne-Illusie's decomposition theorem.  

Let $K_0\subset K$ be the maximal unramified subfield with the ring of integers $V_0$ and $f(t)$ the Eisenstein polynomial satisfying $f(\pi)=0$. Consider $f$ as an element in $P=V_0[[t]]$. Equip $\spec(P)$ with the log structure given by $\N\to P, 1\mapsto t$. Note that ${\bf V}=(\spec(V),1\mapsto \pi)$ is an exact closed subscheme of ${\bf P}=(\spec(P), 1\mapsto t)$ defined by the ideal $(f)$. 
\begin{prop}\label{sufficient condition on illusie's problem}
Let $X$ be a semistable reduction over $V$, endowed with the standard log structure $M_{X_0}$. If the log scheme $(X,M_{X_0})$ over $\bf V$ is liftable to the log scheme $\bf P$, then the log scheme $\bf X_0$ is DR-decomposable (that is Problem \ref{prob_Illusie} holds true for $\bf X_0\to k$).
\end{prop}
\begin{proof}
Notice that $\bf W_2$ is the exact closed subscheme of $\bf P$ defined by the ideal $(p,t)^2$. A lifting of $(X,M_X)$ over $\bf P$ pulls back to a $\bf W_2$-lifting of the special fiber $\bf X_0$.  Thus Theorem \ref{thm_decom_lifting} implies the result.
\end{proof}
The last result of this section enables us to take $k$ to be algebraically closed in the study of Problem \ref{prob_Illusie}, which we shall assume in the next section. 
\begin{cor}\label{cor_change_field}
Let $f: \mathbf{X}\rightarrow \mathbf{k}$ be a smooth morphism of Cartier type. Let $k'$ be a field extension of $k$ which is also perfect.  Let $f_{k'}: \mathbf{X_{k'}}
\to {\bf k'}$ be the base change of $f$ by ${\bf k'}\to {\bf k}$. Then $\tau_{<p}F_{X/k\ast}\Omega^{\bullet}_{\mathbf{X}/\mathbf{k}}$ is decomposable if and only if $\tau_{<p}F_{X_{k'}/k'\ast}\Omega^{\bullet}_{\mathbf{X_{k'}}/{\mathbf{k'}}}$ is decomposable.
\end{cor}
\begin{proof}
By Theorem \ref{thm_decom_lifting}, it is equivalent to show that $\mathbf{X}$ lifts over ${\bf W}_2(k)$ if and only if $\mathbf{X_{k'}}$ lifts over ${\bf W}_2(k')$. One has the following commutative diagram of morphisms of schemes:
$$
\xymatrix{
X_{k'} \ar[r]^{\pi} \ar[d] & X\ar[d]\\
 \spec \ k'\ar[r]& \spec\ k.
}$$
By the flat base change theorem, we know that the composite of the natural maps
$$
\pi^*: H^2(X,T_{\mathbf{X}/\mathbf{k}})\stackrel{}{\longrightarrow}H^2(X,T_{\mathbf{X}/\mathbf{k}})\otimes_kk'\longrightarrow H^2(X_{k'}, T_{\mathbf{X_{k'}}/\mathbf{k'}})
$$
is injective. The same argument using \v{C}ech representatives of the obstruction classes as given in Theorem \ref{thm_decom_lifting} shows that
$$
\pi^*(\omega(f))=\omega(f_{k'}).
$$
Thus $\omega(f)=0$ if and only if $\omega(f_{k'})=0$. The corollary follows.
\end{proof}

\section{Examples}\label{section-example}
In this section we take $k$ to be an algebraically closed field of characteristic $p>0$.  The aim of this section is to construct semistable families over $W=W(k)$ whose special fibers are DR-indecomposable.  Because of the criterion Theorem \ref{thm_decom_lifting}, it is equivalent to construct \emph{semistable families over $W$ whose special fibers are non-liftable over ${\bf W}_2$}. Such examples negate Problem \ref{prob_Illusie}, and consequently the base changes of their special fibers by the absolute Frobenius $F_{\bf k}$ are non-liftable over $(\Spec\ W_2, 1\mapsto p)$.
\subsection{A technical lemma}
Let $X$ be a strictly semistable reduction over $W$. Let ${\bf X_0}$ be the associated log scheme over ${\bf k}$. In the following second lemma, we show that, in constrast to the smoothing effect on the underlying scheme of the given lifting ${\bf X_1}$ of ${\bf X_0}$ over the log  base $(\Spec\ W_2, 1\mapsto p)$, the underying scheme of a lifting of ${\bf X_0}$ over the log base ${\bf W}_2=(\Spec\ W_2, 1\mapsto 0)$ keeps the singularity of  $X_0$.
\begin{lem}\label{lem_w2lifting}
Notation as above. Let $X_0=\bigcup_{i\in I}X_0^i$ be the irreducible decomposition of the scheme $X_0$. Suppose ${\bf X_1}$ be a ${\bf W}_2$-lifting of ${\bf X_1}$. Then the underlying scheme $X_1$ of ${\bf X_1}$ is the schematic union of closed subschemes $X_1=\bigcup_{i\in I} X^i_1$ such that, for each nonempty set $J\subseteq I$ of indices, the schematic intersection $\bigcap_{j\in J}X^j_1$ is a $W_2$-lifting of the scheme $\bigcap_{j\in J}X_0^j$.
\end{lem}
\begin{proof}
Set $\sI_i=I_i+pI_i,$ where $I_i$ is the ideal sheaf of $X_0^i$ in $X_0$. Then, $\sI_i$ is an ideal sheaf of $\sO_{X_1}$. We claim that the closed subschemes $X_1^i$s defined by $\sI_i$s have the property as claimed in the lemma. To show this, it suffices to prove the following properties:
\begin{enumerate}
  \item $\sO_{X_1}/\sI_i$ is flat over $W_2$,
  \item $\bigcap \sI_i=0$, and
  \item for each nonempty $J\subseteq I$, $\sO_{X_1}/\cup_{j\in J}\sI_j$ is flat over $W_2$.
\end{enumerate}
Since $\widehat{\sO_{X_1,x}}$ is faithfully flat over $\sO_{X_1,x}$ for each point $x\in X_1$, it suffices to verify the above claim after tensoring with $\widehat{\sO_{X_1,x}}$ for every $x\in X_1$. By \cite[Theorem 3.5, Proposition 3.14]{KKato1988}, there is an \'etale morphism $U\rightarrow X_1$ such that we have
$$\xymatrix{
U\ar[r]^-f \ar[dr]_{\pi'|_{U}} &\textrm{Spec}(W_2[x_1,\cdots,x_n]/(x_1\cdots x_r)) \ar[d]\\
 & \spec(W_2)
},$$
where $f$ is an \'etale morphism. As a consequence, there is an isomorphism
$$\alpha:\widehat{\sO_{X_1,x}}\cong W_2[[x_1,\cdots,x_n]]/(x_1\cdots x_r)$$ such that each $\sI_i\widehat{\sO_{X_1,x}}$ (whenever it is nonempty) is generated by $\alpha^{-1}(\Pi_{j\in J_i}x_j)$ for some nonempty set $J_i\subseteq\{1,\cdots,r\}$. Moreover, $\{1,\cdots,r\}$ is the disjoint union of $J_i$s. Then the claim follows from direct calculations.
\end{proof}

\subsection{Construction}
The strategy of our construction, which negate Problem \ref{prob_Illusie}, is as follows. Let $Z/W$ be a semistable family and $X/W$ be an admissible blow-up of $Z$ along a regular center $Y_0\subset Z_0$. Among many cases, there are two typical situations under which the log scheme $\bf X_0$ cannot be lifted to $\bf W_2$. 

\begin{description}
	\item[Case I] $Y_0$ is non $W_2$-liftable (Proposition \ref{prop_conterexample2}).\\
	\item[Case II] The pair $(Z_0,Y_0)$ is non $W_2$-liftable.
\end{description}
Although this strategy shall provides complicated semistable families negating Problem \ref{prob_Illusie}, we construct examples only when $Z$ is smooth over $W$. A study of case I provides us examples whose generic fiber is $\mathbb{P}^n$ for $n\geq 5$ (Example 1). This demonstrate a surprising fact that a semistable family of a quite simple variety may have bad hodge theoretic behavior on the central fiber. A study of Case II provides us examples of arbitrary dimension $\geq 2$ (Example 2).
 \begin{prop}\label{simple lemma}
Let $Z$ be a regular scheme which is smooth over $W$, whose special fiber is denoted by $Z_0$. Let $Y_0\subset Z_0$ be a proper smooth closed subvariety. Then the special fiber $X_0$ of the admissible blow-up $X=Bl_{Y_0}Z$ consists of two smooth components $Bl_{Y_0}Z_0$ and $\mathbb{P}(N_{Y_0/Z})$ which intersect transversally along $\mathbb{P}(N_{Y_0/Z_0})$.
\end{prop}
\begin{proof}
The proof is fairly standard and therefore omitted, see e.g Section 5.1 in \cite{Fulton1998}. 
\end{proof}
\noindent {\bf Example 1}
\begin{prop}\label{prop_conterexample2}
Notation as in Proposition \ref{simple lemma}.  Equip $X_0$ with its natural log structure, which makes it into a log scheme over $\bf k$. If $Y_0$ is non $W_2$-liftable, then $\bf X_0$ is non ${\bf W}_2$-liftable.
\end{prop}
\begin{proof}
Assume the contrary that $\bf X_0$ is ${\bf W}_2$-liftable. Then, by Lemma \ref{lem_w2lifting}, the irreducible component $\mathbb{P}(N_{Y_0/Z})$ is $W_2$-liftable. However, a result of Cynk-van Straten (Proposition \ref{cynklemma} below) implies that the base of the natural projection $\mathbb{P}(N_{Y_0/Z})\to Y_0$, which is $Y_0$, is also $W_2$-liftable. 
\end{proof}
For reader's convenience, we include the result of Cynk-van Straten as follows.
\begin{prop}{\cite[Theorem 3.1]{Cynk2009}}\label{cynklemma}
	Let $\pi:Y\rightarrow X$ be a morphism of schemes over $k$ and let $S=\spec\ A$, where $A$ is an artinian local ring with residue field $k$. Assume that $\sO_X =\pi_\ast\sO_Y$ and $R^1\pi_\ast(\sO_Y)=0$. Then for every lifting $\sY\rightarrow S$ of $Y$ as a scheme there exists a lifting $\sX\rightarrow S$ making the
following diagram commutative
	$$\xymatrix{
		Y \ar@{^{(}->}[r]\ar[d] & \sY \ar[d] \\
		X \ar@{^{(}->}[r] & \sX
	}$$
\end{prop}
Proposition \ref{prop_conterexample2} provides the following examples: take a smooth projective variety $Y_0$ over $k$ which is non $W_2$-liftable, and take a closed embedding $Y_0\hookrightarrow Z_0$ over $k$ into a smooth projective variety such that the codimension $\textrm{codim}_{Z_0}Y_0\geq 2$ and $Z_0$ admits a smooth lifting $Z$ over $W$ (for example take $Z_0$ to be a projective space of high dimension). Then $X=\textrm{Bl}_{Y_0}Z$ is a semistable family over $W$ whose special fiber $\bf X_0$ is non ${\bf W}_2$-liftable.\\

\noindent {\bf Example 2}\\

Notice that Mukai \cite{Mukai2013} has obtained a nice generalization to higher dimension of Raynaud's classical example \cite{Raynaud1978} of non $W_2$-liftable smooth projective surface over $k$. His construction, together with an idea of Liedtke-Satriano \cite[Theorem 1.1 (a)]{LM2014}, allows us to produce examples of all relative dimensions $\geq 2$. Let us recall first the following

\begin{defn}[\cite{Mukai2013}]\label{defn_Tango}
A smooth projective curve $C$ over $k$ of genus $\geq2$ is called a Tango-Raynaud curve, if there exists a rational function $f$ on $C$ such that $df\neq0$ and that $(df)=pD$ for some ample divisor $D$.
\end{defn}
A typical example of Tango-Raynaud curve has its affine model defined by the following polynomial
$$G(x^p)-x=y^{pe-1},$$
where $G$ is a polynomial of degree $e\geq 1$ in the variable $x$. The following lemma is well known.
\begin{lem}{\cite[\S2]{Mukai2013}}\label{lem_tango_curve}
Let $C$ be a Tango-Raynaud curve. Then, there exists a rank two vector bundle $E$ on $C$ together with a smooth curve $D$ contained in its projectification $\mathbb{P}_C(E)$, such that the composite of natural maps $D\rightarrow \mathbb{P}_C(E)\rightarrow C$ is the relative Frobenius $F_{D/k}: D\rightarrow D'=C$ over $k$.
\end{lem}

Now we proceed to the last construction in this paper.  We shall use notation in Lemma \ref{lem_tango_curve} in the following
\begin{prop}\label{examples of all dimensions}
Let $C$ be a Tango-Raynaud curve over $k$. Choose and then fix a smooth lifting $\sC$ of $C$ over $W$. Then the vector bundle $E$ can be lifted to a vector bundle over $\sC$. Choose such a lifting $\sE$ of $E$. For each natural number $d\geq 2$, set $Z^d=\mathbb{P}_{\sC}(\sE\oplus \sO_{\sC}^{d-2})$ and $X^d=Bl_{D}Z^d$, where $D$ is a closed scheme of $Z_d$ by the natural inclusions
$$D\subset \mathbb{P}_{C}(E)\subset \mathbb{P}_{C}(E\oplus \sO_C^{d-2})\subset Z_d.$$
Then, the special fiber of the semistable family $X^d$ over $W$ is non ${\bf W}_2$-liftable.
\end{prop}
\begin{proof}
We prove the statement for $d=2$ only (the proof for $d\geq 3$ is the same). Set
	$$
	C_0=C, \quad Y_0=D, \quad Z_0=\mathbb{P}_C(E), \quad Z=Z_2.
	$$
Assume the contrary that the special fiber of $Bl_{Y_0}Z$ is ${\bf W}_2$-liftable. It follows from Lemma \ref{lem_w2lifting} that the pair $(Z_0, Y_0)$ consisting of the component $Z_0=Bl_{Y_0}Z_0$ of $X_0$ together with the divisor $Y_0=\mathbb{P}(N_{Y_0/Z})\cap Z_0\subset X_0$ lift to a pair $(Z_1,Y_1)$ over $W_2$ (The scheme $Z_1$ is not necessarily the mod $p^2$-reduction of $Z$). On the other hand, Proposition \ref{cynklemma} implies that the projection $Z_0\to C_0$ is the reduction of a certain $W_2$-morphism $Z_1\to C_1$. Therefore, the composite $F_0: Y_0\hookrightarrow Z_0\to C_0$ lifts to the composite $F_1: Y_1\hookrightarrow Z_1\to C_1$ over $W_2$. But this leads to a contradiction: the nonzero morphism $dF_1: F_1^*\Omega_{C_1/k}\to \Omega_{Y_1/k}$ is divisible by $p$ and it induces a nonzero morphism over $k$
	$$
	\frac{dF_1}{p}:F_0^*\Omega_{C_0/k}\to \Omega_{Y_0/k},
	$$
	which is impossible because of the degree. Therefore, $\bf X_0$ is indeed non ${\bf W}_2$-liftable as claimed.
\end{proof}
\appendix
\section{Preliminarires on semistable reductions}
In this appendix we present some facts on semistable reductions. They are more or less standard material. We collect them here for the convenience of our readers.
\begin{defn}{\cite[\S 2.4]{deJong1996}}
Let $V$ be a regular Noetherian scheme. Let $D\subset V$ be a divisor of $X$ and  $D_i\subset D$, $i\in I$ be its irreducible components (considered as reduced closed subschemes). We say that $D$ is a \textit{strict normal crossing divisor} if the following conditions hold:
	\begin{enumerate}
		\item $D$ is reduced, i.e. $D=\bigcup_{i\in I} D_i$ (scheme-theoretically),
		\item For any nonempty subset $J\subset I$, $D_J:=\bigcap_{j\in J}D_j$ is a regular subscheme of codimension $\sharp J$ in $S$.
	\end{enumerate}
A divisor $D\subset V$ is called \textit{normal crossing} if there is a surjective \'etale morphism $V'\to V$ such that the scheme-theoretic inverse image of $D$ is a strict normal crossing divisor on $X'$.
\end{defn}

\begin{defn}
Let $A$ be a local ring and $x_1,\dots,x_r\in m_A$. We say that $x_1,\dots,x_r$ form a part of a parameter system if their images in $m_A/m_A^2$ are linearly independent over $A/m_A$.
\end{defn}
\begin{lem}\label{lem_SNC_algebra}
Let $A$ be a Noetherian regular local ring and $x_1,\dots,x_r\in m_A$. Then the subscheme defined by $x_1x_2\cdots x_r$ is a strict normal crossing divisor if and only if $x_1,\dots,x_r$ form a part of a parameter system.
\end{lem}
\begin{proof}
See \cite[Theorem 14.2]{Matsumura1986}.
\end{proof}

\begin{lem}\label{lem_reg_lc}
Let $A$ be a Noetherian regular local ring and $I\subset m_A$ be an ideal such that $A/I$ is a regular local ring. Then there are $x_1,\dots,x_r\in m_A$, $r=\dim A-\dim A/I$ such that 
	\begin{enumerate}
		\item $x_1,\dots,x_r$ form a part of a parameter system, and
		\item $I=(x_1,\dots,x_r)$.
	\end{enumerate}
\end{lem}

\begin{proof}
	By \cite[Theorem 21.2]{Matsumura1986}, $I=(x_1,\dots,x_r)$ for an $A$-regular sequence $x_1,\dots,x_r\in m_A$. By \cite[Theorem 14.2]{Matsumura1986}, the images of such $x_1,\dots,x_r$ in $m_A/m_A^2$ must be linearly independent over $A/m_A$.
\end{proof}
The above lemmas suggest the following
\begin{defn}\label{defn_NC_subscheme}
Let $V$ be a regular Noetherian scheme and $D\subset V$ a normal crossing divisor. We say that $D$ has strict normal crossings with a regular closed subscheme $Y\subset V$ if \'etale locally there exist $x_1,\dots,x_k$ such that
	\begin{enumerate}
		\item $x_1,\dots,x_k$ form a part of a parameter system,
		\item $D=\{x_1\cdots x_r=0\}$ for some $1\leq r\leq k$,
		\item $Y=\{x_s=0,\dots, x_k=0\}$ for some $1\leq s\leq k$.
	\end{enumerate}
\end{defn}

\begin{defn}{\cite[\S 2.16]{deJong1996}}\label{defn_ssreduction}
Let $R$ be a discrete valuation ring. Let $K$ denote its fractional field and $k$ the residue field. Then an $R$-scheme $Z$ is a \textit{semistable reduction} over $\spec(R)$ if  the following two properties hold:
	\begin{enumerate}
		\item the generic fiber $Z_{K}=Z\times_R{K}$ is smooth over $K$,
		\item the special fiber $Z_{0}=Z\times_Rk$ is a normal crossing divisor of $Z$.
	\end{enumerate}
If furthermore $Z_0$ is a strict normal crossing divisor of $Z$, we say that $Z$ is a \textit{strictly semistable reduction}. A (strictly) semistable family over $\Spec(R)$ is a \emph{proper}  (strictly) semistable reduction over $\Spec(R)$.
\end{defn}
\begin{rmk}\label{rmk_ssred_local}
One has the isomorphism
$$X\simeq \spec (\hat{R}[[x_1,\dots,x_n]]/(x_1\cdots x_r-\pi))\quad\textrm{(\cite[\S 2.16]{deJong1996})}$$ formal locally. As a consequence, the log morphism $(X,X_0)\to (\spec(R),\spec(k))$ of log schemes is log smooth \cite[Examples 3.7 (2)]{KKato1988}.
\end{rmk}

Notations as in Definition \ref{defn_ssreduction}. Let $Y\subset Z_0$ be a regular closed subscheme such that $Z_0$ has normal crossings with $Y$. Let $D\subset Z$ be a (strict) normal crossing divisor such that $D\cup Z_0$ is a normal crossing divisor of $Z$ and has normal crossings with $Y$. Denote $X={\rm Bl}_YZ$ for the blow-up of $Z$ along $Y$. Since the blow-up procedure commutes with \'etale base changes, The above lemmas show that 
\begin{lem}
Notations as above. $X$ is a semistable reduction over ${\rm Spec}(R)$. The inverse image of $Z_0\cup D$ is a normal crossing divisor on $X$. If moreover $Z_0\cup D$ is strict normal crossing, then its inverse image is strict normal crossing.
\end{lem}
\begin{lem}\label{lem_induction_ssreduction}
Assume the following conditions:
	\begin{enumerate}
		\item Let $R$ be a discrete valuation ring and $A$ a Noetherian regular local ring over $R$ such that $Z:={\rm Spec}(A)\to {\rm Spec}(R)$ is a strict semistable reduction.
		\item Let $D$ be a divisor on $Z$ such that $Z_0\cup D$ is strict normal crossing. 
		\item Let $Y\subset Z_0$ be a regular closed subscheme such that $Z_0$ has strict normal crossings with $Y$ and $Y$ is contained in every component of $D$.
		\item the number of components of $Z_0\cup D$ containing $Y$ is less than $\dim X-\dim Y$.
	\end{enumerate}
Then there is a regular divisor $D'$ on $Z$, containing $Y$, such that $D'$ is a semistable reduction over $R$, $Z_0\cup D\cup D'$ is strict normal crossing.
\end{lem}
\begin{proof}
	By condition (3) there exist $x_1,\dots,x_k\in A$ such that
	\begin{itemize}
		\item $x_1,\dots,x_m$ form a part of a parameter system,
		\item $D=\{x_1\cdots x_r=0\}$ for some $1\leq r\leq m$,
		\item $Y=\{x_1=0,\dots, x_m=0\}$.
	\end{itemize}
By Lemma \ref{lem_SNC_algebra}, we have
	$$\pi=y_1\cdots y_s$$
	where $y_1,\dots,y_s\in A$ form a part of a parameter system. Since $Y\subset Z_0$, $(y_1,\dots,y_s)/m_A^2$ define a subspace of $(x_1,\dots,x_m)/m_A^2$. Let $s\in m_A$ and $\bar{s}$ be the image of $s$ in $m_A/m_A^2$. Denote $D'$ for the divisor on $Z$ defined by $s$. By condition (4), one has
$$
(y_1,\dots,y_s,x_1,\dots,x_r)\cap(x_1,\dots,x_m)/m_A^2\varsubsetneqq (x_1,\dots,x_m)/m_A^2.
$$
Therefore we can choose a sufficiently general $s\in(x_1,\dots,x_m)$ so that $\bar{s}\neq 0$ does not lies in $(y_1,\dots,y_s,x_1,\dots,x_r)/m_A^2$. This ensures that $D'$ is regular, flat over $\spec(R)$, and $Z_0\cup D\cup D'$ is strict normal crossing.  As a consequence, $D'_0=Z_0\cap D'$ (scheme theoretic) is strict normal crossing. Hence $D'$ is a semistable reduction over $R$.
\end{proof}
The remaining part of the appendix is devoted to construct the log version of the residue sequence, which is used in the proof of proposition \ref{prop_KEY}.

Let $R$ be a discrete valuation ring and $X$ be a semistable reduction over $S=\spec(R)$. Let $0\in\spec(S)$ be the closed point and $X_0$ the fiber over $0$. Let $D\subset X$ be a connected regular divisor such that $D$ is a semistable reduction over $S$ and $X_0\cup D$ is a normal crossing divisor. By Remark \ref{rmk_ssred_local}, the morphisms of log schemes $\mathbf{X}=(X,X_0)\to \mathbf{S}=(S,0)$, $\mathbf{X}'=(X,X_0\cup D)\to \mathbf{S}$ and $\mathbf{D}=(D,D\cap X_0)\to\mathbf{S}$ are log smooth. Denote $\sM_{X_0\cup D}$ for the log structure on $X$ induced from $X_0\cup D$, then $\sM_{X_0\cup D}|_D$ is induced from the pre-log structure $\sM_{D_0}\oplus \mathbb{N}$. As a consequence,
$$\Omega^1_{(D,\sM_{X_0\cup D}|_D)}\simeq \Omega^1_{\mathbf{D}/\mathbf{S}}\oplus\sO_D$$ and thus
\begin{align}\label{align_split_cotangent}
\Omega^i_{(D,\sM_{X_0\cup D}|_D)}\simeq \Omega^i_{\mathbf{D}/\mathbf{S}}\oplus \Omega^{i-1}_{\mathbf{D}/\mathbf{S}}.
\end{align}
Define the residue map
$${\rm res}_{\mathbf{D}}:\Omega^i_{\mathbf{X}'/\mathbf{S}}\to \Omega^{i-1}_{\mathbf{D}/\mathbf{S}}$$
to be the composite of the natural map
$$\Omega^i_{\mathbf{X}'/\mathbf{S}}\to \Omega^i_{(D,\sM_{X_0\cup D}|_D)/\mathbf{S}}$$
with the projection 
$$\Omega^i_{(D,\sM_{X_0\cup D}|_D)}\stackrel{(\ref{align_split_cotangent})}{\simeq} \Omega^i_{\mathbf{D}/\mathbf{S}}\oplus \Omega^{i-1}_{\mathbf{D}/\mathbf{S}}\to \Omega^{i-1}_{\mathbf{D}/\mathbf{S}}.$$
Recall that the log morphism $\mathbf{X}'\to\mathbf{X}$ induces a natural map
$$\Omega^i_{\mathbf{X}/\mathbf{S}}\to\Omega^i_{\mathbf{X}'/\mathbf{S}}$$
\begin{lem}\label{lem_residue_sequence}
	The sequence
	$$0\to\Omega^i_{\mathbf{X}/\mathbf{S}}\to\Omega^i_{\mathbf{X}'/\mathbf{S}}\stackrel{{\rm res}_{\mathbf{D}}}{\to}\Omega^{i-1}_{\mathbf{D}/\mathbf{S}}\to0$$
	and its restriction
	$$0\to\Omega^i_{\mathbf{X}/\mathbf{S}}|_{X_0}\to\Omega^i_{\mathbf{X}'/\mathbf{S}}|_{X_0}\to\Omega^{i-1}_{\mathbf{D}/\mathbf{S}}|_{D_0}\to0$$
	are short exact sequences.
\end{lem}
\begin{proof}
	The proof is done by local calculations. We may assume that $R$ is complete,
	$$X\simeq \spec (R[[x_1,\dots,x_n]]/(x_1\cdots x_r-\pi)), \quad r<n$$
	and $D$ is defined by $x_{r+1}=0$.  Then
	$$\Omega^1_{\mathbf{X}/\mathbf{S}}\simeq \bigoplus_{i=1}^r\sO_X\frac{dx_i}{x_i}/\sO_X\sum_{i=1}^r\frac{dx_i}{x_i}\oplus\bigoplus_{j=r+1}^n\sO_Xdx_j,$$
	$$\Omega^1_{\mathbf{X}'/\mathbf{S}}\simeq \bigoplus_{i=1}^{r+1}\sO_X\frac{dx_i}{x_i}/\sO_X\sum_{i=1}^r\frac{dx_i}{x_i}\oplus\bigoplus_{j=r+2}^n\sO_Xdx_j$$
	and
	$$\Omega^1_{\mathbf{D}/\mathbf{S}}\simeq \bigoplus_{i=1}^r\sO_D\frac{dx_i}{x_i}/\sO_D\sum_{i=1}^r\frac{dx_i}{x_i}\oplus\bigoplus_{j=r+2}^n\sO_Ddx_j.$$
	The residue map is defined as follows.
	For every $\alpha\in\Omega^i_{\mathbf{X}'/\mathbf{S}}$, there is a unique decomposition
	$$\alpha=\beta+\gamma\wedge\frac{dx_{r+1}}{x_{r+1}}$$
	where $\beta$, $\gamma$ do not involve $\frac{dx_{r+1}}{x_{r+1}}$. Then 
	$${\rm res}_{\mathbf{D}}(\alpha)=\gamma|_{D}.$$
	These calculations show that
	$$0\to\Omega^i_{\mathbf{X}/\mathbf{S}}\to\Omega^i_{\mathbf{X}'/\mathbf{S}}\to\Omega^{i-1}_{\mathbf{D}/\mathbf{S}}\to0$$
	is a short exact sequence.
	By restriction we have an exact sequence
	$$\Omega^i_{\mathbf{X}/\mathbf{S}}|_{X_0}\to\Omega^i_{\mathbf{X}'/\mathbf{S}}|_{X_0}\to\Omega^{i-1}_{\mathbf{D}/\mathbf{S}}|_{D_0}\to0.$$
	By local calculations the map on the left is injective. This proves the lemma.
\end{proof}

\textbf{Acknowledgment:} We would like to thank sincerely Luc Illusie for several valuable e-mail communications, and Weizhe Zheng for his more conceptual argument to Theorem \ref{thm_decom_lifting} which is also included in the note. Warm thanks go to Christian Liedtke for his interest and comments on the early version of this paper. We thank heartily Mircea Musta\c{t}\u{a} for the explaination on \cite[Proposition 31.3]{Mustata} which is very helpful to the proof of Proposition \ref{prop_KEY}. We thank heartily Dan Abramovich for reminding us of the work \cite{Abr2016}. Last but not least, the first named author would like to thank Kang Zuo for his interest and constant support. The second named author would like to express his gratitude to Xiaotao Sun and Jun Li for their constant encouragement throughout the work.

\end{document}